\documentclass[10pt,reqno]{amsart}
\usepackage[colorlinks=true,allcolors=blue,backref=page]{hyperref}
\usepackage{color}
\usepackage{amsmath, amssymb, amsthm}

\usepackage{mathrsfs}
\usepackage{mathtools}
\usepackage[noabbrev,capitalize,nameinlink]{cleveref}
\crefname{equation}{}{}
\usepackage{fullpage}
\usepackage[noadjust]{cite}
\usepackage{graphics}
\usepackage{pifont}
\usepackage{tikz}
\usepackage{bbm}
\usepackage[T1]{fontenc}

\usetikzlibrary{arrows.meta}

\usepackage{environ}
\usepackage{framed}
\usepackage{url}
\usepackage[linesnumbered,ruled,vlined]{algorithm2e}
\usepackage[noend]{algpseudocode}
\usepackage[labelfont=bf]{caption}
\usepackage{cite}
\usepackage{framed}
\usepackage[framemethod=tikz]{mdframed}
\usepackage{appendix}
\usepackage{graphicx}
\usepackage[textsize=tiny]{todonotes}
\usepackage{tcolorbox}
\usepackage{enumerate}
\usepackage[shortlabels]{enumitem}
\allowdisplaybreaks[1]

\apptocmd{\sloppy}{\hbadness 10000\relax}{}{} 

\crefname{algocf}{Algorithm}{Algorithms}

\crefname{equation}{}{} 
\crefname{conjecture}{Conjecture}{Conjectures} 
\AtBeginEnvironment{appendices}{\crefalias{section}{appendix}} 

\crefformat{enumi}{#2#1#3}
\crefrangeformat{enumi}{#3#1#4 to~#5#2#6}
\crefmultiformat{enumi}{#2#1#3}%
{ and~#2#1#3}{, #2#1#3}{ and~#2#1#3}

\usepackage[color,final]{showkeys} 

\colorlet{refkey}{orange!20}
\colorlet{labelkey}{blue!30}

\crefname{algocf}{Algorithm}{Algorithms}

\numberwithin{equation}{section}
\newtheorem{theorem}{Theorem}[section]

\newtheorem{lemma}[theorem]{Lemma}

\crefname{claim}{Claim}{Claims}

\newtheorem{corollary}[theorem]{Corollary}

\newtheorem*{question*}{Question}

\theoremstyle{definition}

\newtheorem*{definition*}{Definition}

\theoremstyle{remark}
\newtheorem{remark}[theorem]{Remark}

\newcommand{\snorm}[1]{\lVert#1\rVert}

\newcommand{\sang}[1]{\langle #1 \rangle}

\newcommand{\mb}{\mathbb}
\newcommand{\mbf}{\mathbf}
\newcommand{\mbm}{\mathbbm}
\newcommand{\mc}{\mathcal}

\newcommand{\ol}{\overline}
\newcommand{\on}{\operatorname}

\newcommand{\wt}{\widetilde}

\newcommand{\eps}{\varepsilon}

\allowdisplaybreaks

\title{Convergent points for random power series on the unit circle}

\author[A1]{Marcus Michelen}
\address{Department of Mathematics, Northwestern University}
\email{michelen.math@gmail.com, michelen@northwestern.edu}

\author[A2]{Mehtaab Sawhney}
\address{Department of Mathematics, Columbia University, New York, NY 10027}
\email{m.sawhney@columbia.edu}

\begin{document}

\begin{abstract} 
Consider a random power series of the form 
$P(z) = \sum_{n\ge 1} \eps_n a_n z^{n}$
where $a_n \in \mb{C}$ are deterministic and $\eps_n$ are chosen independently and uniformly at random from  $\{\pm 1\}$.  Kolmogorov's three-series theorem states that if $\sum_{n} |a_n|^2 = \infty$ then $P(z)$ almost-surely diverges at almost every $z$ with $|z| = 1$.  Dvoretzky and Erd\H{o}s proved in 1959 that if $|a_n| = \Omega(1/\sqrt{n})$ then in fact $P$ almost surely diverges at \emph{every} $|z| = 1$.  Erd\H{o}s then asked in 1961 if this is sharp, meaning that if $|a_n| = o(1/\sqrt{n})$ then there is almost surely some convergent point $z$ with $|z| = 1$.  We prove this in a strong sense and show that if $a_n = o(1/\sqrt{n})$ then in fact the set of convergent points of $P$ with $|z| = 1$ has Hausdorff dimension $1$.  
\end{abstract}

\maketitle

\section{Introduction}\label{sec:introduction}
Consider a random power series \[P(z) = \sum_{n \geq 1} \eps_n a_n z^n \]
where $a_n \in \mb{C}$ are deterministic and $\eps_n$ are chosen independently and uniformly from $\{\pm1\}$.  A special case of Kolmogorov's three-series theorem shows that for a fixed $z$ with $|z| = 1$, $P(z)$ converges almost-surely if and only if $\sum_{n \ge 1} |a_n|^2 < \infty$.  Fubini's theorem then allows one to show that if $\sum_{n \ge 1} |a_n|^2 < \infty$, then almost-surely $P$ converges at almost-every point with $|z| = 1$.  Seeking to better understand convergence of random  series, a classic result of 
Salem and Zygmund \cite[Theorem~5.1.5]{SZ54} asserts that if $\sum_{n \geq 1}|a_n|^2<\infty$ and $R_k = \sum_{n \geq k}|a_n|^2$ satisfies 
\[\sum_{k\ge 1}^{\infty}\frac{R_k^{1/2}}{k(\log k)^{1/2}}<\infty\]
then $P$ is almost-surely uniformly convergent on the unit circle.  Further, under mild conditions on the coefficient sequence, this result is if and only if.

The purpose of this paper is to consider the case when we expect that the power series will not converge uniformly on the unit circle. Dvoretzky and Erd{\H o}s \cite{DE59} proved that if there exists a positive sequence $c_n$ with 
\[\limsup_{n\to\infty}\frac{\sum_{j=0}^{n}c_j^2}{\log(1/c_n)} > 0\]
and that $|a_n|\ge c_n$ then almost surely the series $\sum_{n\ge 1}\eps_n a_nz^{n}$ diverges \emph{everywhere} on the unit circle $|z| = 1$. An important special case of this result is that when $|a_n|\ge 1/\sqrt{n}$, then almost surely the series $\sum_{n\ge 1}\eps_n a_nz^{n}$ diverge everywhere on the unit circle $|z| = 1$. This paper is concerned with the converse of the result of Dvoretzky and Erd{\H o}s \cite{DE59}. In particular, Erd\H{o}s \cite[pg~254]{Erd61} (see also \cite[Problem~\#527]{BloWeb}) asked if when $|a_n| = o(1/\sqrt{n})$ is it true that almost all series $\sum_{n\ge 1}\eps_n a_nz^{n}$ have a point on $|z| = 1$ which is convergent. (Erd\H{o}s in particular writes that ``Perhaps every series satisfying $\cdots$ has this property.'')

Our main result confirms the conjecture of Erd\H{o}s.
\begin{theorem}\label{thm:main} 
Let $a_n\in \mb{C}$ with $\lim_{n\to \infty}n^{1/2} |a_n| = 0$ and let $\eps_n$ be chosen independently and uniformly at random from $\{\pm 1\}$. Then almost surely there exists $z$ with $|z| = 1$ such that
$\sum_{n=1}^{\infty}\eps_na_nz^{n}$
converges.
\end{theorem}

Our methods in fact deliver a substantially stronger conclusion regarding the set of points on the unit circle which converge. 
\begin{theorem}\label{thm:dim}
Let $a_n\in \mb{C}$ with $\lim_{n\to \infty}n^{1/2} |a_n| = 0$ and let $\eps_n$ be chosen independently and uniformly at random from $\{\pm 1\}$. Then almost surely the set of points with $|z| = 1$  such that $\sum_{n=1}^{\infty} \eps_na_nz^{n}$ converges has Hausdorff dimension $1$.
\end{theorem}

In other words, if one considers a sequence $a_n$ with $\sum_{n \geq 1} |a_n|^2 = \infty$ but $\lim_{n \to \infty} n^{-1/2} |a_n| = 0$, then the set of convergent points of $P$ on the unit circle has Lebesgue measure $0$ but Hausdorff dimension $1$.

\subsection{Outline of proof of \texorpdfstring{\cref{thm:main}}{Theorem 1.1}}
The result of Dvoretzky and Erd{\H o}s \cite{DE59}---which built on earlier work of Dvoretzky \cite{Dvo56} concerning covering the unit circle with random arcs---shows that if the coefficient sequence is $a_n = n^{-1/2}$, then the series $\sum_{n = 1}^{\infty} a_n \eps_n z^{n}$ diverges \emph{everywhere} on the unit circle. We now sketch a short proof of this special case, which is closely related to the argument in \cite{DE59} and serves as motivation for much of our analysis. 

Suppose that for a fixed $z$, the series $\sum_{n = 1}^{\infty} a_n \eps_n z^{n}$ converges. Then for every $\eps > 0$, there exists $N_{\eps}$ such that for $N \geq N_{\eps}$ we have that the event 
\begin{equation}\label{eq:huer}
\mc{C}_N(z) = \left\{\inf_{N\le j\le N^2}\Big|\sum_{n=N}^{j}\eps_n a_n z^{n}\Big|\le \eps\right\}
\end{equation}
holds.  Note that the random variable $Z = \sum_{n=N}^{N^2} \eps_n a_n z^{n}$ has variance $\mb{E}[|Z|^2] = \sum_{n=N}^{N^2} |a_n|^2 \approx \log N$. Thus the event in \eqref{eq:huer} can be compared to a Brownian motion staying inside the interval $[-\eps,\eps]$ for time proportional to $\log N$. This happens with probability $\lesssim e^{-\Omega(\eps^{-2}\log N)} = N^{-\Omega(\eps^{-2})}$. Moreover, for $z$ and $z'$ with, e.g,  $|z - z'| \leq N^{-5}$ , the quantity in \eqref{eq:huer} changes by $o(1)$ and so there are essentially only $N^{O(1)}$ distinct points $z$ relevant to \eqref{eq:huer}.  A union bound then shows divergence at all points on the unit circle.

Recall that \cref{thm:main} states that if the coefficients decay even slightly faster than $n^{-1/2}$, then the series converges.  Already, revisiting the event $\mc{C}_N(z)$ tells us why one could hope for this behavior: if the coefficients $a_n$ now satisfy $|a_n| \leq \delta / \sqrt{n}$, then the variance of the sum is now of order $\asymp \delta^2 \log N$, implying that the probability a given $z$ satisfies $\mc{C}_N(z)$ is more like $N^{-\Theta(\delta^2 / \eps^2)}$.  A back-of-the-envelope calculation to estimate the correlations of the sum in \eqref{eq:huer} suggests that, $\sum_{n = N}^{N^2} \eps_n a_n z^n$ and $\sum_{n = N}^{N^2} \eps_n a_n w^n$ typically decorrelate once we have $|z - w| \gg N^{-2}$. Provided $\delta \ll \eps$, one  expects that in fact \emph{many} points satisfy \eqref{eq:huer}.   

Of course, for a point $z$ to be a convergent point for $\sum_n  \eps_n a_n z^n$, it is not enough for \eqref{eq:huer} to hold for some large $N$.  We must also have, for instance, the same holds when we look between coefficients in $[N^2 ,N^4]$, and $[N^4,N^8]$ and so on.  We saw $\mb{P}(\mc{C}_N(z)) \approx N^{-\Theta(\delta^2/\eps^2)}$ and hope that $\mc{C}_N(z), \mc{C}_N(w)$ are nearly independent for $|z - w| \gg N^{-2}.$  For any given point $z$ that satisfying $\mc{C}_N(z)$, it is in fact unlikely for it to then satisfy, say, $\mc{C}_{N^2}(z)$.  However, we now suddenly have more points to work with, as the events $\mc{C}_{N^2}(z)$ and $\mc{C}_{N^2}(w)$ are now approximately independent for $|z - w| \gg N^{-4}$, meaning we have roughly $N^{4}$ many chances for some $\mc{C}_{N^2}$ to hold while we only had around $N^2$ chances for $\mc{C}_N$ to hold.  Further, if $|\zeta - z| \ll N^{-2}$, then if $\mc{C}_N(z)$ holds it is quite likely that $\mc{C}_N(\zeta)$ holds.  This leads to a branching process sort of argument: look at points that satisfy $\mc{C}_N(z)$; think of their ``children'' as points at distance $\ll N^{-2}$; show that if we have many points satisfying $\mc{C}_N(z)$ then we have many points satisfying $\mc{C}_{N^2}(w)$. 

This is a close approximation of our argument, with two tweaks.  The first is to choose our scales to grow a bit slower than $N,N^2,N^4,\ldots.$  We choose a sequence of increasing scales $N_1 < N_2 < \cdots$ such that $(\log N_i)^{\omega(1)}\le N_{i+1}/N_i\le N_i^{o(1)}$ (see \cref{eq:M-N-definition}); there is significant flexibility in this choice.  We also need a slightly more sophisticated event than in \eqref{eq:huer}, as satisfying \eqref{eq:huer} is not enough to guarantee convergence.  Defining $\delta_i = \max_{n \geq N_i} \sqrt{n} |a_n|$ we say that a point $z$ is $i$-good if 
\begin{align}\label{eq:goodness}
\sup_{N_{i-1} \le j \le N_{i}} \Big|\sum_{n = N_{i-1}}^{j} \eps_n a_n z^{n}\Big| \le \delta_{i-1}^{1/2},\, \qquad \text{ and } \qquad 
\Big|\sum_{n = N_{i-1}}^{N_{i}-1} \eps_n a_n z^{n}\Big| \le i^{-2}.
\end{align}
If a point $z$ is $i$-good for all $i$, then one can show that the partial sums of $\sum_{n\ge 1} \eps_n a_n z^{n}$ form a Cauchy sequence (see \cref{lem:convergence-criterion}). 

The key observation at this stage is that, on their own, the first event in \cref{eq:goodness} holds with probability $(N_{i}/N_{i-1})^{-\Omega(\delta_{i-1})}$, while the second event holds with probability $\gtrsim (\log N_{i})^{-\Omega(1)}$ (using that $i \lesssim \log N_{i}$). Provided $\delta_i$ does not decay too slowly, the first event dominates the convergence criterion. This dominance is precisely what makes the Dvoretzky–Erd{\H o}s \cite{DE59} upper bound sharp.
The second event is nonetheless crucial to establish convergence: since $\delta_i$ decays arbitrarily slowly, the series $\sum \delta_{i}^{\Omega(1)}$ need not converge. 

To make this heuristic picture rigorous, we first apply a Lindeberg argument to replace the Rademacher coefficients by Gaussian ones (\cref{lem:lindberg-pairs}). With Gaussian coefficients, the Gaussian correlation inequality allows us to decouple the two events in \cref{eq:goodness}, and also to decouple the real and imaginary parts in the first event of \cref{eq:goodness}.  We are then left with events concerning Gaussian processes that are easily handled by classical tools. This is carried out in \cref{subsec:one-pt}.

To upgrade this one-point expectation argument into a full convergence proof, we employ a branching process formalism together with a second moment method. In particular, if a point $z$ is $i$-good, then standard derivative bounds imply that any $z'$ with $|z-z'|\le N_i^{-1}(\log N_i)^{-\Omega(1)}$ is also $i$-good. To formalize this, we let $\mc{A}_i$ be the set of points in $\{j/N_i : j \in [N_i]\}$ that are good for all scales less than $i$. We then refine each point $j/N_i$ into approximately $(N_{i+1}/N_i)(\log N_i)^{-\Omega(1)}$ distinct points at the next scale, and test whether they are $(i+1)$--good. The key step is to show that if $|\mc{A}_i|\ge N_i^{1-\Omega(1)}$, then 
\[
|\mc{A}_{i+1}|\;\;\ge\;\;(N_{i+1}/N_i)^{1-\Omega(1)} \cdot |\mc{A}_i|\;\;\ge\;\; N_{i+1}^{1-\Omega(1)}.
\]

This is established via a second moment analysis. A subtle point is that even distant points can be strongly correlated. For example, if $a_n = 0$ whenever $n \equiv 1 \pmod{2}$, then the partial sums at $z$ and $-z$ coincide. However, the large sieve inequality (\cref{thm:large-siev}) ensures that such strongly correlated pairs are rare. Since we must control not only the full segment sums but also all partial sums, we consider correlations across various partial intervals; nevertheless, by a basic reduction, for each interval $[N_{i-1},N_i]$ it suffices to check only $(\log N_i)^{O(1)}$ partial sums (see \cref{subsec:conv}). This part of the argument is developed in \cref{subsec:corr-point,subsec:comp}.

The final portion of the paper concerns upgrading \cref{thm:main} to \cref{thm:dim}. The main technical difference from \cref{thm:main} is that we must ensure not only that $|\mc{A}_i|$ is large, but also that many points in $\mc{A}_i$ have many descendants.  We will use Frostman's lemma (\cref{lem:Frostman}) to lower bound the Hausdorff dimension of the set of convergent points and so we need to construct a sufficiently rich ``Frostman measure'' on the set of convergent points.  The branching structure employed to ensure convergence allows us to construct such a measure in a similar manner to how one constructs the Cantor measure: one iteratively pushes the measure downward from $i$-good points to $i+1$-good points.  The analysis becomes slightly more technical---for instance when applying the second moment method we must localize to the descendants of a specific node in the branching process---but our toolkit remains the same. This analysis is carried out in \cref{sec:Hausdorff}.

\begin{remark}
    In the case of $a_n = 1/\sqrt{n}$, Dvoretzky and Erd\H{o}s proved that almost  surely $P$ diverges everywhere on the unit circle.  Our analysis shows that this case is extremely delicate.  In fact, our proof shows that if one sets $L$ to be a large constant, then the probability there is a point $z$ with $|z| = 1$ so that $\max_{\ell} \left|\sum_{n = 1}^\ell \eps_n a_n z^n \right| \leq L$ is $1 - o_{L \to \infty}(1)$.  To see this, replace the $\delta_{k-1}^{1/2}$ in the definition of $k$-good at \eqref{eq:good-definition} with $L$ and note the bound in \cref{lem:one-point} becomes $\exp(-\Omega( (L^{-1}\log\log N_{k-1})^2))$.  In the proof of \cref{thm:main}, this lower bound is only used in \eqref{eq:good-pt-LB-use}, which for $L$ large enough then shows the inductive step to establish a proliferation of points with this weaker definition of good.
\end{remark}

\subsection{Notation and organization of the paper}
We use standard asymptotic notation throughout. In particular, $A\lesssim B$ means there exists an absolute constant $C>0$ such that $A\lesssim C B$. Similarly $A\asymp B$ means $A\lesssim B$ and $B\lesssim A$.  For $\theta \in \mb{R}/\mb{Z}$, we write $e(\theta) = e^{2\pi i \theta}$.  

In \cref{subsec:conv} we formalize our convergence criterion for our partial sums and note that it suffices to consider a sparse sequence of scales. In \cref{subsec:one-pt}, we prove that the probability a particular point is good is appropriate. In \cref{subsec:corr-point}, we use the large sieve to bound the number of correlated pairs and show under non-correlation the events that $z$ and $z'$ are good are uncorrelated. The proof of \cref{thm:main} is then completed in \cref{subsec:comp}. Finally we carry out the proof of \cref{thm:dim} in \cref{sec:Hausdorff}.

\subsection{Acknowledgments}
The authors thank Dmitrii Zakharov for useful clarifications regarding Hausdorff dimension and Oren Yakir for useful discussions. 

MM is supported in part by NSF CAREER grant DMS-2336788 as well as DMS-224662.
This research was conducted during the period MS served as a Clay Research Fellow. A portion of this research was conducted while MS was visiting Oxford University.

\section{Proof of \texorpdfstring{\cref{thm:main}}{Theorem 1.1}}

We fix a set of coefficients $a_n$ with $\lim_{n\to \infty}n^{-1/2} |a_n| = 0$ and $C\ge 1$ will be a sufficiently large absolute constant.  We first define the relevant scales. Define sequences $N_{i}$ and $M_i$ inductively. Set $N_1 = M_1$ to be the smallest constant such that $N_1\ge C$ and $|a_n| \cdot n^{-1/2}\le 1/C$ for all $n\ge N_1$.  Inductively define  
\begin{equation} \label{eq:M-N-definition}
M_{j+1} = \lfloor \log N_j\rfloor^{\lfloor \log\log N_j \rfloor}\qquad \text{ and } \qquad N_{j+1} = N_j \cdot M_{j+1}.
\end{equation}
By taking a crude bound we have that \begin{equation*} 
    \log N_j \asymp j(\log j)^2\,.
\end{equation*}
We next define 
\begin{equation*}
\delta_k = \max\big\{(\log k)^{-1/2}, \max_{n\ge N_k}n^{-1/2} \cdot |a_n|\big\}.
\end{equation*}
Note that by assumption we have that $\lim_{k\to \infty}\delta_k= 0$. 

\subsection{Sparsifying the set of scales: a criterion for convergence} \label{subsec:conv}

The goal of this subsection is to work on a sparser set of scales in order to come up with a criterion for convergence.   Our main goal will be to prove \cref{lem:convergence-criterion}, which states that under two almost-sure events, we can identify convergent points by checking certain events at our sparse scale. 

We break the interval $[N_i,N_{i+1}]$ into a series of smaller scales which will be used for subsequent parts of the proof. Let $r_{i,1} = N_i$ and define $r_{i,\ell + 1} = r_{i,\ell} (1 + \eta_{i,\ell})$ with $\eta_{i,\ell}\in [(\log N_i)^{-10}, 2\cdot (\log N_i)^{-10}]$ and with a final index $\ell_i$ such that $r_{i,\ell_i} = N_{i+1}$. Note that $\ell_i \asymp (\log N_i)^{10} \cdot (\log\log N_i)^2$ by construction. 
We first prove that it suffices to consider partial sums at indices of the form $r_{i,j}$.  To this end, define \begin{equation*}
    \mc{E}_{1,k} =  \left\{\sup_{j \in [\ell_k]} \sup_{\substack{\ell\in [N_k,N_{k+1}]\\ r_{k,j}\le \ell<r_{k,j+1}}}\sup_{|z| = 1}\Big|\sum_{n = r_{k,j}}^{\ell} \eps_n a_n z^{n}\Big|\ge (\log N_k)^{-2}\right\}\,.
\end{equation*}

\begin{lemma}\label{lem:sparsify-to-r}
    Set $\mc{E}_1 = \{\mc{E}_{1,k} \text{ occurs for infinitely many }k\}$.  Then $\mb{P}(\mc{E}_1) = 0$.
\end{lemma}
\begin{proof}
    We will apply the Borel--Cantelli lemma to the collection of events $\mc{E}_{1,k}$. Noting that $\ell_k N_{k+1}^2 \leq  N_k^2 (\log N_k)^2,$ it suffices to prove that for a fixed $\ell\in [N_k,N_{k+1}]$ with $r_{k,j}\le \ell <r_{k,j+1}$ we have that
    \[\mb{P}\Big[\sup_{|z| = 1}\Big|\sum_{n = r_{k,j}}^{\ell}\eps_na_nz^{n}\Big|\ge (\log N_k)^{-2}\Big]\le N_k^{-4}.\]
    We then note that the derivative of $\sum_{n = r_{k,j}}^{\ell}\eps_na_nz^{n}$ is $\sum_{n=r_{k,j}}^{\ell}\eps_na_n\cdot n z^{n-1}$. Note by triangle inequality that $|\sum_{n=r_{k,j}}^{\ell}\eps_na_n \cdot nz^{n-1}|\le \sum_{n=1}^{\ell}n\le \ell^2\le N_k^4$.  If we set $\mc{N}$ to be a $N_k^{-5}$ net of size $\lesssim N_k^{5}$ it is sufficient to prove 
    \[\max_{z \in \mc{N}} \mb{P}\Big[\Big|\sum_{n = r_{k,j}}^{\ell}\eps_na_nz^{n}\Big|\ge (\log N_k)^{-2}/2\Big]\le N_k^{-10}.\]
    Note that $\sum_{n=r_{k,j}}^{\ell} |a_n|^2\le \sum_{n=r_{k,j}}^{\ell} 1/n\lesssim (\log N_k)^{-10}$.  The desired result is then an immediate consequence of the Azuma--Hoeffding inequality. 
\end{proof}

\cref{lem:sparsify-to-r} will allow us to only consider partial sums indexed by the terms $\{r_{k,j}\}$.  With this in mind, for the interval from $[r_{k,j}, r_{k,j+1} - 1]$ define \begin{equation*}
    Q_{k,j}(z) =  \sum_{n=r_{k,j}}^{r_{k,j+1} - 1}\eps_n a_nz^{n}.
\end{equation*}
On the $k$th scale we will consider points in the following net:
\begin{equation*}
    \mc{S}_k = \left\{\frac{j}{N_k} : j \in [N_k] \right\} \subset \mb{R}/\mb{Z}\,.
\end{equation*}

Now that we have restricted to our sparser scales given by $r_{i,j}$, we are ready to define our central events. 
At step $k$, we define the set of ``$k$-good'' points $\mc{G}_k$ points via 
\begin{equation}\label{eq:good-definition}
    \mc{G}_k = \left\{\theta \in \mb{R}/\mb{Z} : \sup_{j\in [\ell_{k-1}]}\Big|\sum_{r = 1}^{j} Q_{k-1,r}(e(\theta))\Big|\le 3\cdot \delta_{k-1}^{1/2}\text{ and }\Big|\sum_{r = 1}^{\ell_{k-1}} Q_{k-1,r}(e(\theta))\Big|\le 3 \cdot k^{-2} \right\}\,.
\end{equation}
We note that the event of a point lying in $\mc{G}_k$ depends only on the random variables $\eps_n$ for $n \in [N_{k-1},N_k - 1]$

The set of ``alive'' points at step $k$ is defined as \begin{equation}\label{eq:A-k-def}
    \mc{A}_k = \left\{\theta \in \mc{S}_k : \theta \in \mc{G}_{k} \wedge \exists~\varphi \in \mc{A}_{k-1} \text{ with }|\varphi - \theta| \leq N_{k-1}^{-1} (\log N_{k-1})^{-8} \right\}\,.
\end{equation} 
Finally, the set of points that remain alive is \begin{equation}
    \mc{A} = \bigcap_{k \geq 1} \left(\mc{A}_k + [-N_k^{-1}(\log N_k)^{-5}, N_k^{-1}(\log N_k)^{-5}] \right) \subseteq \mb{R}/\mb{Z}\,.
\end{equation}

Note that $\mc{A}$ is no longer a discrete set of points.  The idea will be that under basic control of the derivative, bounds for an angle $\theta \in \mc{A}_k$ can be translated to the whole interval $\theta+ [-N_k^{-1}(\log N_k)^{-5}, N_k^{-1}(\log N_k)^{-5}]$.  With this in mind, we define \begin{equation}
\mc{E}_{2,k} = \left\{\sup_{|z| = 1} \sup_{N_k\le \ell<N_{k+1}-1}\Big|\sum_{j=N_{k}}^{\ell}ja_j \cdot \eps_jz^{j}\Big|\ge N_{k+1} \cdot \log N_{k+1}\right\}\,.
\end{equation}

\begin{lemma}\label{lem:deriv-controlled}
    Set $\mc{E}_2 = \{\mc{E}_{2,k} \text{ occurs for infinitely many }k\}$.  Then $\mb{P}(\mc{E}_2) = 0.$
\end{lemma}
\begin{proof}
    A simple bound on  the second derivative is $\sum_{j\le N_{k+1}}j^2\lesssim N_{k+1}^3$.  Setting up another use of Borel--Cantelli, it is sufficient  to take $z$ in a $N_{k+1}^{-4}$ separated net $\mc{N}$ of size $\lesssim N_{k+1}^4$ and prove that 
    \begin{equation}\label{eq:deriv-too-big}
    \max_{z \in \mc{N}} \mb{P}\left(\sup_{N_k\le \ell<N_{k+1}-1}\Big|\sum_{j=N_{k}}^{\ell}ja_j \cdot \eps_jz^{j}\Big|\ge N_{k+1} \cdot (\log N_{k+1})/2\right) \leq N_{k+1}^{-6}.\end{equation}
     Note that $(\sum_{j\le N_{k+1}}|ja_j|^2)^{1/2}\le (\sum_{j\le N_{k+1}}j)^{1/2}\le N_{k+1}$ and so Azuma--Hoeffding proves \eqref{eq:deriv-too-big}.  Applying the Borel--Cantelli lemma completes the proof.
\end{proof}

We are now ready to prove our convergence criterion.

\begin{lemma}\label{lem:convergence-criterion}
    On the event $\mc{E}_1^c \cap \mc{E}_2^c$, the power series $P(e(\theta))$  converges for every $\theta \in \mc{A}$.
\end{lemma}
\begin{proof}
    We will prove that the sequence of partial sums is Cauchy.  Let $K_0$ be so that value so that for all $k \geq K_0$ the event   $\mc{E}_{1,k} \cup \mc{E}_{2,k}$ does not hold.  Let $K \geq K_0$ be large enough.  Consider $N_K \leq \ell_1 < \ell_2$ and let $k_1, k_2$ be so that $N_{k_r} \leq \ell_r < N_{k_r + 1}$ for $r \in \{1,2\}$.    Write \begin{align}\label{eq:show-cauchy}
        \sum_{j = \ell_1}^{\ell_2} \eps_j a_j  z^j = \sum_{j = N_{k_1}}^{N_{k_2} - 1} \eps_j a_j z^j - \sum_{j = N_{k_1}}^{\ell_1-1} \eps_j a_j z^j  + \sum_{j = N_{k_2}}^{\ell_2} \eps_j a_j z^j\,.
    \end{align}

    The second two terms will be handled by the same argument and so we handle the first term.  First note that we may find $t$ so that $r_{k_1,t} \leq \ell_1 - 1 < r_{k_1,t+1}$.  Since $\mc{E}_{1,k_1}^c$ holds we have that \begin{equation}\label{eq:replace-rk}
    \left|\sum_{j = N_{k_1}}^{\ell_1-1} \eps_j a_j z^j\right| \leq \left|\sum_{j = 1}^{t-1} Q_{k_1,j} (z) \right| + (\log N_{k_1})^{-2}\,.
    \end{equation}
    Since $\theta \in \mc{A}$ we have there is some $\varphi \in \mc{A}_{k_1 + 1}$ so that if we set $w = e(\varphi)$ then we have  $|\theta - \varphi| \leq N_{k_1 + 1}^{-1} (\log N_{k_1+1})^{-5}$.  Writing $z = e(\theta)$, since the event $\mc{E}_{2,k_1}^c$ holds we can bound   
    \begin{equation}\label{eq:approx-with-deriv}
    \left|\sum_{j = 1}^{t-1} Q_{k_1,j} (z)\right| \leq \left|\sum_{j = 1}^{t-1} Q_{k_1,j} (w)\right| + (\log N_{k_1 + 1})^{-4}\,.
    \end{equation}
    Since $\varphi \in \mc{A}_{k_1 + 1}$ and $w = e(\varphi)$, we have
    \begin{equation}\label{eq:approx-with-Q}
        \left|\sum_{j = 1}^{t-1} Q_{k_1,j} (w)\right| \leq  3 \delta_{k_1}^{1/2}
    \end{equation}
    Combining \cref{eq:replace-rk}, \cref{eq:approx-with-deriv,eq:approx-with-Q} and applying the same bounds for the third term in \eqref{eq:show-cauchy} shows
    \begin{equation}\label{eq:sum-intermediate}
        \left|\sum_{j = N_{k_1}}^{\ell_1-1} \eps_j a_j z^j\right| + \left|\sum_{j = N_{k_2}}^{\ell_2} \eps_j a_j z^j\right| \leq 10\delta_{k_1}^{1/2}\,.
    \end{equation}
    To handle the first term in \eqref{eq:show-cauchy}, first write \begin{equation*}
        \sum_{j = N_{k_1}}^{N_{k_2} - 1} \eps_j a_j z^j = \sum_{k = k_1}^{k_2 - 1} \sum_{j = 1}^{\ell_k} Q_{k,j}(z)\,. 
    \end{equation*} 
    We note that \eqref{eq:approx-with-deriv} shows that for the same $\varphi \in \mc{A}_{k+1}$ and $w = e(\varphi)$ we have \begin{equation*}
        \left|\sum_{j = 1}^{\ell_k} Q_{k,j}(z)\right| = \left|\sum_{j = 1}^{\ell_k} Q_{k,j}(w) \right| + O(\log N_k^{-2}) \leq \frac{4}{k^2}   
    \end{equation*}
    where in the last inequality we used that $\varphi \in \mc{G}_{k+1}$ and that $\log N_k \geq k$. Combining the previous two displayed equations shows \begin{equation}\label{eq:sum-endpoints}
        \left| \sum_{j = N_{k_1}}^{N_{k_2} - 1} \eps_j a_j z^j \right| \leq \sum_{k \geq k_1} \frac{4}{k^2} \leq \delta_{K}
    \end{equation}
    Combining \eqref{eq:show-cauchy} with \eqref{eq:sum-intermediate} and \eqref{eq:sum-endpoints} and recalling that $\delta_K \to 0$ as $K \to \infty$ completes the proof.
\end{proof}

\subsection{A one point-estimate: the probability a point is good}\label{subsec:one-pt}

The main purpose of this subsection is to lower bound the probability that a point lies in the ``good'' set $\mc{G}_k$ defined at \eqref{eq:good-definition}.  For technical reasons, we will need a bit more wiggle room.
\begin{lemma}\label{lem:one-point}
    For $z$ with $|z| = 1$ we have $$\mb{P}\left[\sup_{j\in [\ell_{k-1}]}\Big|\sum_{r = 1}^{j} Q_{k-1,r}(z)\Big|\le \delta_{k-1}^{1/2} \wedge \Big|\sum_{r = 1}^{\ell_{k-1}} Q_{k-1,r}(z)\Big|\le  k^{-2}\right] \geq \exp\left( -\Omega\left(\delta_{k-1} \cdot (\log\log N_{k-1})^2 \right)\right)\,.$$
\end{lemma}

We will prove these by Gaussian comparison and so we begin with a basic Lindeberg--type bound.  We prove a sufficiently general version that will be useful for our second moment calculation in \cref{subsec:corr-point}.

\begin{lemma}\label{lem:lindberg-pairs}
There exists $L\ge 1$ such that the following holds. 
Let $a_i\in \mb{C}, b_i \in \mb{C}$, $\eps_i \sim \{\pm 1\}$ and $g_i\sim \mc{N}(0,1)$. Then for any function $f:\mb{C}^2\to \mb{C}$, we have that  
\[\bigg|\mb{E}\Big[f\Big(\sum_{i=1}^{\ell}a_i\eps_i,\sum_{i=1}^{\ell}b_i\eps_i\Big)\Big] - \mb{E}\Big[f\Big(\sum_{i=1}^{\ell}a_ig_i,\sum_{i=1}^{\ell}b_ig_i\Big)\Big]\bigg|\le L\cdot \max_{|\alpha| = 3} \snorm{\partial^\alpha f}_{\infty}\cdot \left(\sum_{j=1}^{\ell}(|a_j|^3 +|b_j|^3)  \right).\]
\end{lemma}
\begin{proof}
For $j \in \{0,\ldots,\ell\}$ define the random variable \[X_j = f\Big(\sum_{i=1}^{\ell - j }a_i\eps_i + \sum_{i  = \ell - j + 1}^{\ell} a_i g_i,\sum_{i=1}^{\ell - j }b_i\eps_i + \sum_{i  = \ell - j + 1}^{\ell} b_i g_i\Big)\]
then note that we are interested in $|\mb{E} X_0 - \mb{E} X_\ell|$.  By telescoping we have \begin{equation} \label{eq:lindeberg-telescope}
    |\mb{E} X_0 - \mb{E} X_\ell| \leq \sum_{j = 1}^\ell |\mb{E} X_{j - 1} - \mb{E} X_j| \leq \sum_{j = 1}^\ell \sup_{z,w} | \mb{E}[f(z + a_j \eps_j,w + b_j \eps_j) -\mb{E}[f(z + a_j g_j,w + b_j g_j) ]|\,.
\end{equation}
By Taylor's theorem, for any $\theta$ we have \begin{align}
     f(z + a_j \theta,w + b_j \theta) &= f(z,w) + \theta\left( a_j \partial^{1,0}f(z,w) + b_j\partial^{0,1}f(z,w)\right) \nonumber\\
     &\qquad +  \frac{\theta^2}{2}\left(a_j^2 \partial^{2,0} f(z,w) + 2a_jb_j \partial^{1,1} f(z,w) + b_j^2 \partial^{0,2} f(z,w)  \right) \nonumber \\
     &\qquad + O(|\theta|^3 (|a_j|^3 + |b_j|^3) \max_{|\alpha| = 3} \snorm{\partial^\alpha f}_\infty ) \label{eq:taylors-theorem}
\end{align}
Since $\mb{E} \eps_j = \mb{E} g_j$ and $\mb{E} \eps_j^2 = \mb{E} g_j^2$, \eqref{eq:taylors-theorem} implies \[ \left|\mb{E}[f(z + a_j \eps_j,w + b_j \eps_j) -\mb{E}[f(z + a_j g_j,w + b_j g_j)\right| \lesssim (|a_j|^3 + |b_j|^3) \max_{|\alpha| = 3} \snorm{\partial^\alpha f}_\infty\,.  \]
Combining with \eqref{eq:lindeberg-telescope} completes the proof.
\end{proof}

We will apply \cref{lem:lindberg-pairs} to approximately couple $Q_{k,j}$ with its gaussian counterpart defined via \begin{equation}\label{eq:wtq-def}
\wt{Q}_{k,j}(z) =  \sum_{n=r_{k,j}}^{r_{k,j+1} - 1}g_n a_nz^{n}
\end{equation}
where $g_n\sim \mc{N}(0,1)$. 

\begin{corollary}\label{cor:gaussian-approx}
For $k$ large enough, we may couple $Q_{k,j}(z)$ and $\wt{Q}_{k,j}(z)$ such that $|Q_{k,j}(z) - \wt{Q}_{k,j}(z)|\lesssim N_k^{-1/12}$ with probability $1-N_k^{-1/12}$.
\end{corollary}
\begin{proof}
Recall that the Wasserstein $1$-metric  between two probability measures $(\mu,\nu)$ is defined as 
\[W_1(\mu,\nu) = \inf_{(x,y)\sim \Gamma(\mu,\nu)}(\mb{E}[|x-y|]);\]
here $(x,y)\sim \Gamma(\mu,\nu)$ is any coupling of $\mu$ and $\nu$. Since the measures $\mu,\nu$ live in $\mb{R}^2$ (which we may identify with $\mb{C}$) and applying the dual representation of $W_1$ (see e.g. \cite[(5.11)]{Vil09}), we have that 
\[W_1(\mu,\nu) = \sup_{\substack{f:\mb{C}\to \mb{R}\\\snorm{f}_{\on{Lip}}\le 1}}\Big|\mb{E}_{x\sim \mu}[f(x)] - \mb{E}_{x\sim \nu}[f(x)]\Big|.\]
Let $\psi:\mb{C} \to \mb{R}$ be a smooth non-negative bump function with $\int \psi = 1$ and $\psi(z) = 0$ for $|z| \geq 1$.  Set $\psi_\eps(z) = \eps^{-2} \psi(z / \eps).$  For a $1$-Lipschitz function $f$ if we write $f_\eps = f \ast \psi_\eps$ then note that $\max_{|\alpha| =3} |\partial^\alpha f_\eps| \lesssim \eps^{-2} $.  By \cref{lem:lindberg-pairs} we then have \begin{align*}
    \Big|\mb{E}_{x\sim \mu}[f(x)] - \mb{E}_{x\sim \nu}[f(x)]\Big| \lesssim \Big|\mb{E}_{x\sim \mu}[f_\eps(x)] - \mb{E}_{x\sim \nu}[f_\eps(x)]\Big| + \eps \lesssim  \frac{\eps^{-2}}{N_k^{1/2}} + \eps \lesssim N_k^{-1/6}
\end{align*}
by taking $\eps = N_k^{-1/6}$. 
By the coupling definition of the $W_1$ metric, there exists a coupling $\Gamma(\mu,\nu)$ such that 
\[\mb{E}_{(x,y)\sim \Gamma(\mu,\nu)}[|x-y|]\lesssim N_k^{-1/6}.\]
By Markov's inequality, this implies that $|x-y|\lesssim N_k^{-1/12}$ with probability $N_k^{-1/12}$. 
\end{proof}

Since we are dealing with an intersection of Gaussians lying in convex sets, we will make use of the Gaussian Correlation Inequality of Royen \cite{Roy14} (see \cite{LM17}). 
\begin{theorem}\label{thm:GCI}
Let $K,L\subseteq \mb{R}^d$ be symmetric convex sets. Let $g\sim \mc{N}(0,I_d)$ be a standard Gaussian. Then 
\[\mb{P}[g\in K\cap L] \ge \mb{P}[g\in K] \cdot \mb{P}[g\in L].\]
\end{theorem}

We begin by lower bounding the second event in \eqref{eq:good-definition}.

\begin{lemma}\label{lem:endpoint-LB}
    We have $$\mb{P}\left(\left|\sum_{r = 1}^{\ell_k} \wt{Q}_{k,r}(z) \right| \leq 2^{-1} k^{-2} \right) \gtrsim  \delta_k^{-2}\cdot (\log N_k)^{-5} \,.$$
\end{lemma}
\begin{proof}
    Note that $\sum_{n\in [N_k,N_{k+1})}|a_n|^2\le \sum_{n\in [N_k,N_{k+1})}\delta_k^2/n\le 2\delta_k^2 \cdot (\log\log N_i)^2$.  We note that $\sum_{r = 1}^{\ell_k} \wt{Q}_{k,r}(z)$ is a complex Gaussian $Z$ with $\mb{E} |Z|^2 \leq 2\delta_k^2 \cdot (\log\log N_k)^2$  and so 
        \[\mb{P}\left(\left|\sum_{r = 1}^{\ell_k} \wt{Q}_{k,r}(z) \right| \leq 2^{-1} k^{-2} \right)\gtrsim \frac{k^{-4}}{\delta_k^2 \cdot (\log\log N_k)^2}\gtrsim \delta_k^{-2}\cdot (\log N_k)^{-5}.\qedhere \]
\end{proof}

We now handle the first event in \eqref{eq:good-definition}.
\begin{lemma} \label{lem:brownian-LB}
We have
\[\mb{{P}} \left(
    \sup_{j\in [\ell_k]}\Big|\sum_{r = 1}^{j} \wt{Q}_{k,r}(z)\Big|\le 2^{-1} \cdot \delta_k^{1/2} \right) \geq \exp\left(- \Omega(\delta_k\cdot (\log\log N_k)^2) \right)\,.\]
\end{lemma}
\begin{proof}
    Define \begin{equation*}
        A_1 = \left\{\sup_{j\in [\ell_k]}\Big|\sum_{r = 1}^{j} \on{Re}(\wt{Q}_{k,r}(z))\Big|\le 2^{-2} \cdot \delta_k^{1/2}\right\}\quad \text{ and } \quad A_2 = \left\{\sup_{j\in [\ell_k]}\Big|\sum_{r = 1}^{j} \on{Im}(\wt{Q}_{k,r}(z))\Big|\le 2^{-2} \cdot \delta_k^{1/2}\right\}\,.
    \end{equation*}
    We lower bound \begin{equation}\label{eq:Q-brownian-LB}
    \mb{{P}} \left(
    \sup_{j\in [\ell_k]}\Big|\sum_{r = 1}^{j} \wt{Q}_{k,r}(z)\Big|\le 2^{-1} \cdot \delta_k^{1/2} \right) \geq \mb{P}(A_1 \cap A_2) \geq \mb{P}(A_1) \mb{P}(A_2)\end{equation}
    where the second inequality is by the Gaussian Correlation Inequality \cref{thm:GCI} after confirming that the sets $A_1, A_2$ are symmetric convex sets in the Gaussian variables $\{g_n\}$.  Each of the two can be handled in the same manner.  We note that $ \left\{\Big|\sum_{r = 1}^{j} \on{Re}(\wt{Q}_{k,r}(z))\Big|\right\}_{j \in [\ell_k]}$
    is a mean-zero Gaussian process with variance at most $\sum_{n = N_k}^{N_{k + 1} - 1} |a_n|^2 \leq 2 \delta_k^2 (\log \log N_k)^2\,.$
    As such, we may find deterministic times $0 \leq t_1 < \ldots < t_{\ell_k} \leq 2 \delta_k^2 (\log \log N_k)^2$ so that if we set $(B_t)_{t \geq 0}$ to be standard Brownian motion then  $$ \left\{\Big|\sum_{r = 1}^{j} \on{Re}(\wt{Q}_{k,r}(z)) \Big|\right\}_{j \in [\ell_k]} = (B_{t_j})_{j \in [\ell_k]}$$
    where the equality is in distribution. We then have \begin{align*}
        \mb{P}(A_1) &\geq \mb{P}( \max_{0 \leq t \leq 2 \delta_k^2 (\log \log N_k)^2} |B_t| \leq 2^{-2}\delta_k^{1/2}) =\mb{P}( \max_{0 \leq t \le 1} |B_t| \leq 2^{-5/2}\delta_k^{-1/2}(\log\log N_k)^{-1}) \\
        &\geq \exp\left(-\Omega( \delta_k \cdot (\log\log N_k)^2)\right)
    \end{align*}
    where the last inequality is via the reflection principle (see \cite[pg.~342]{Fel71}) which implies that in the limit $a\to 0$ we have that $\mb{P}(\max_{0 \leq t \le 1}|B_t|\le a)\approx \frac{4}{\pi} \exp\big(\frac{-\pi^2}{8a^2}\big)$.
    Combining with \eqref{eq:Q-brownian-LB} completes the proof.
\end{proof}

We are now ready to prove our one-point estimate \cref{lem:one-point}:
\begin{proof}[Proof of \cref{lem:one-point}]
    By \cref{cor:gaussian-approx} we may bound \begin{align*}
    \mb{P}&\left[ \Big|\sum_{r = 1}^{\ell_{k-1}} Q_{k-1,r}(z)\Big|\le  k^{-2} \wedge \sup_{j\in [\ell_{k-1}]}\Big|\sum_{r = 1}^{j} Q_{k-1,r}(z)\Big|\le \delta_{k-1}^{1/2} \right] \\
    &\qquad\geq \mb{P}\left(\left|\sum_{r = 1}^{\ell_{k-1}} \wt{Q}_{k-1,r}(z) \right| \leq 2^{-1}k^{-2}  \wedge 
    \sup_{j\in [\ell_{k-1}]}\Big|\sum_{r = 1}^{j} \wt{Q}_{k-1,r}(z)\Big|\le 2^{-1} \cdot \delta_{k-1}^{1/2} \right) - N_{k-1}^{-1/12} \\
    &\qquad\geq \mb{P}\left(\left|\sum_{r = 1}^{\ell_i} \wt{Q}_{k-1,r}(z) \right| \leq 2^{-1}k^{-2}  \right) \mb{P}\left( 
    \sup_{j\in [\ell_i]}\Big|\sum_{r = 1}^{j} \wt{Q}_{k-1,r}(z)\Big|\le 2^{-1} \cdot \delta_{k-1}^{1/2} \right) - N_{k-1}^{-1/12} \\
    &\qquad\geq \exp\left(- \Omega(\delta_{k-1}\cdot (\log\log N_{k-1})^2) \right)
    \end{align*}
    where the second inequality uses \cref{thm:GCI} and the third uses \cref{lem:endpoint-LB} and \cref{lem:brownian-LB}.
\end{proof}

\subsection{Correlation of points}\label{subsec:corr-point}

We now handle the correlations across pairs of points on the unit circle.  The two main goals of this section are to first prove that most pairs of points in $\mc{S}_k$ exhibit decorrelation on scale $k$ (\cref{lem:corr-point}) and then to prove that the event that two decorrelated points are both ``good'' approximately factors (\cref{lem:decouple}).

The first step in our analysis is to define a coupling between a pair of complex Gaussians and a corresponding independent pair provided certain ``approximate'' orthogonality conditions holds. This is essentially a standard computation.  Throughout this section, for a function  $f:\mb{C}^{t}\to \mb{R}$ we define 
\[\snorm{f}_{\on{Lip}} = \sup_{\substack{(w_1,\ldots,w_t) \in \mb{C}^t\\ (w_1',\ldots,w_t') \in \mb{C}^t}}\frac{f((w_1,\ldots,w_t)) - f((w_1',\ldots,w_t'))}{\big(\sum_{j=1}^{t}|w_j - w_j'|^2\big)^{1/2}};\]
this is the standard definition of Lipschitzness by identifying $\mb{C}^{t}$ with $\mb{R}^{2t}$ and taking the Euclidean norm.

\begin{lemma}\label{lem:coupling}
Consider $\vec{a} = (a_1,\ldots,a_n),~\vec{b} = (b_1,\ldots, b_n) \in \mb{C}^{n}$ with $\sum_{i=1}^{n}|a_i|^2 = \sum_{i=1}^{n}|b_i|^2\le 1$ and $f:\mb{C}^2\to \mb{R}$. Then 
\[\Big|\mb{E}_{\vec{g} \sim \mc{N}(0,1)^{n}}f(\vec{a}\cdot \vec{g}, \vec{b} \cdot \vec{g}) - \mb{E}_{\vec{g},\vec{g}' \sim \mc{N}(0,1)^{n}}f(\vec{a}\cdot \vec{g}, \vec{b} \cdot \vec{g}')\Big|\lesssim \max_{\substack{v \in \{\on{Re}(a),\on{Im}(a)\}\\w\in \{\on{Re}(b), \on{Im}(b)}}|\sang{v,w}|^{1/4} \cdot \snorm{f}_{\on{Lip}}.\]
\end{lemma}
\begin{proof}
First rescale so that $\snorm{f}_{\on{Lip}}\le 1$.  Note that we can rescale so that $\sum_{i = 1}^n |a_i|^2 = 1$ by setting $\alpha^2 = \sum_{i = 1}^n |a_i|^2$ and taking $F(x,y) = \alpha^{-1} f(\alpha x, \alpha y)$.   
We let $v_1 = \on{Re}(\vec{a})$, $v_2 = \on{Im}(\vec{a})$, $v_3 = \on{Re}(\vec{b})$ and $v_4 = \on{Im}(\vec{b})$.  It will be sufficient to prove \begin{equation} \label{eq:gaussian-decor-need}
\Big|\mb{E} f(\vec{a}\cdot \vec{g}, \vec{b} \cdot \vec{g}) - \mb{E} f(\vec{a}\cdot \vec{g}, \vec{b} \cdot \vec{g}')\Big|\lesssim \tau^{1/2} + \tau^{-1/2} \cdot \Big(\sum_{{i\in \{1,2\}, j\in \{3,4\}}}|\sang{v_i,v_j}|^2\Big)^{1/2}
\end{equation}
for each $0\le \tau\le 1/2$.

We will proceed in different cases depending on how two-dimensional the complex random variables $\vec{a} \cdot \vec{g}$ and $\vec{b} \cdot \vec{g}$ each are.  We say that a vector $\vec{a}$ is $\tau$-balanced if $$\inf_{x^2 + y^2 = 1} \|x v_1 + yv_2\|_2^2 \geq \tau\,.$$

If $\snorm{\vec{a}}_2^2\le \tau$, then we are immediately done as  
\begin{equation} \label{eq:zero-out-a}
\Big|\mb{E}f(\vec{a}\cdot \vec{g}, \vec{b} \cdot \vec{g}) - \mb{E}f(\vec{a}\cdot \vec{g}, \vec{b} \cdot \vec{g}')\Big|\le \Big|\mb{E}f(0, \vec{b} \cdot \vec{g}) - \mb{E}f(0, \vec{b} \cdot \vec{g}')\Big| + 2 \cdot \mb{E}|\vec{a} \cdot \vec{g}|\le 4\cdot \tau^{1/2}
\end{equation}
establishing \eqref{eq:gaussian-decor-need}.  If $\vec{a}$ is not $\tau$-balanced, then fix $(x,y)$ such that $\snorm{xv_1 + yv_2}_2^2\le \tau$. Define $\wt{f}(z_1,z_2) = f((y+xi)(z_1),z_2)$; note that $\wt{f}$ is $1$-Lipschitz. Note that $((v_1 + v_2i) \cdot \vec{g}) \cdot (y + xi) = (v_1y - v_2x) \cdot \vec{g} + (v_1x + v_2y) \cdot \vec{g}$. Therefore by replacing $f$ by $\wt{f}$, if $\vec{a}$ is not $\tau$--balanced, at the cost of an $O(\tau^{1/2})$ error we may assume that $v_2 = 0$ by an identical argument to \eqref{eq:zero-out-a}.  If additionally then $\snorm{v_1}_2^2\le \tau$ as above then we are done again by \eqref{eq:zero-out-a}.  Applying the same argument to $\vec{b}$, we now have that each is either $\tau$-balanced or are purely real with appropriate length.  We consider the case where $\vec{a}$ and $\vec{b}$ are each $\tau$-balanced as the other case is strictly simpler.

We now assume $\vec{a}$ is $\tau$-balanced. 
Set $V = \on{span}\{v_1,v_2\} \subset \mb{R}^n$ and let $P_V$ and $P_{V^\perp}$ denote the projections onto $V$ and $V^\perp$. Decompose $\vec{g} = \vec{g}_V + \vec{g}_{V^\perp}$ where $\vec{g}_V \sim N(0,P_V)$ and $\vec{g}_V \sim N(0,P_{V^\perp})$ are independent. Let $\vec{g}_V'$ denote an independent copy of $\vec{g}$.  Note then that in distribution we have \begin{align*}
    (\vec{a} \cdot \vec{g}, \vec{b} \cdot \vec{g}) =   \left( \vec{a} \cdot \vec{g}_V  , \vec{b} \cdot \vec{g}_V +  \vec{b} \cdot \vec{g}_{V^\perp}  \right) \quad \text{ and } \quad 
    (\vec{a} \cdot \vec{g}, \vec{b} \cdot \vec{g}') =   \left( \vec{a} \cdot \vec{g}_V  , \vec{b} \cdot \vec{g}_V' +  \vec{b} \cdot \vec{g}_{V^\perp}  \right)
\end{align*}
and so \begin{align}
    \left|\mb{E}f(\vec{a} \cdot \vec{g}, \vec{b} \cdot \vec{g}) - \mb{E}f(\vec{a} \cdot \vec{g}, \vec{b} \cdot \vec{g}') \right| &\leq \mb{E}|\vec{b} \cdot \vec{g}_V - \vec{b} \cdot \vec{g}_V'| \leq \left( \mb{E}|\vec{b} \cdot \vec{g}_V - \vec{b} \cdot \vec{g}_V'|^2\right)^{1/2} = \sqrt{2} \left( \mb{E}|\vec{b} \cdot \vec{g}_V|^2\right)^{1/2} \nonumber \\ 
    &= \sqrt{2} \left( |P_V v_3|^2 + |P_V v_4|^2\right)^{1/2}\,. \label{eq:CS-projection}
\end{align}
Since $v_1,v_2$ are vectors with $\|v_1\|^2 + \|v_2\|^2 = 1$ that are $\tau$-balanced we have \[|P_Vw|^2 \le \tau^{-1} ( \langle v_1 , w \rangle^2 + \langle v_2, w \rangle ^2).\] 
To see this, consider the Gram matrix 
\[G = \begin{pmatrix}\sang{v_1,v_1} & \sang{v_1,v_2}\\
\sang{v_2,v_1} & \sang{v_2,v_2}\end{pmatrix};\]
the assumed $\tau$-balanced condition is precisely that $s_{\min}(G)\ge \tau$. This implies that $G^{-1}\preceq  \tau^{-1}I$. Via projecting, we may assume that $w = \alpha_1v_1 + \alpha_2v_2$ and $\alpha = (\alpha_1,\alpha_2)$. Then 
\[\snorm{P_Vw}_2^2 = \snorm{\alpha_1v_1 + \alpha_2v_2}_2^2 = \alpha^TG\alpha = (\alpha^T G^T)G^{-1}(G\alpha)\le \tau^{-1} \cdot \snorm{G\alpha}_2^2 = \tau^{-1} ( \langle v_1 , w \rangle^2 + \langle v_2, w \rangle ^2)\]
as desired. Combining  with \eqref{eq:CS-projection} establishes \eqref{eq:gaussian-decor-need}.  
\end{proof}

We next state a routine lemma to bound derivatives of products of functions with bounded derivatives.
\begin{lemma}\label{lem:lip}
Let $f_i:\mb{C}^{k} \to \mb{R}$ be $1$-bounded with $\max_{|\alpha| \leq 3} \|f_j^{\alpha} \|_{\infty} \leq M$ for $M \geq 1$. Then 
\[\max_{|\alpha| = 3}\snorm{\partial^\alpha f_1 \cdots f_t}_{\infty}\le t^{3} M^3.\]
\end{lemma}
\begin{proof}
By the product rule we have \begin{align*}
    \partial^\alpha (f_1 \cdots f_t) = \sum_{\alpha_1 + \cdots + \alpha_t = \alpha} \binom{\alpha}{\alpha_1,\ldots,\alpha_t} \prod_{j = 1}^t \partial^{\alpha_j} f_j\,.
\end{align*}
Noting there are at most $t^3$ sums and at most $3$ derivatives in the product completes the proof.
\end{proof}

We use a version of the large sieve inequality due to Montgomery and Vaughn \cite{MV73}:
\begin{theorem}\label{thm:large-siev}
Let $a_{M+1},\ldots,a_{M+n}$ be complex numbers. Then $\Theta_1,\ldots, \Theta_R$ be phases in $\mb{R}/\mb{Z}$ such that 
\[\sum_{r=1}^{R}\bigg|\sum_{j=M+1}^{M+n}a_j\cdot e(j\Theta_r)\bigg|^2\le (M + \delta^{-1})\cdot \sum_{j=M+1}^{M+n}|a_j|^2\]
where $\delta = \min_{k\neq \ell}\snorm{\Theta_k -\Theta_{\ell}}_{\mb{R}/\mb{Z}}.$
\end{theorem}

With the large sieve inequality in mind, we say a set $\mc{S} = \{\Theta_j\}_{j} \subset \mb{R}/\mb{Z}$ is $\delta$ separated if $\min_{k \neq \ell}\|\Theta_k - \Theta_\ell\|_{\mb{R}/\mb{Z}} \geq \delta $.  Our key application of the large sieve inequality is in the following proposition. 
\begin{lemma}\label{lem:corr-point}
Let $M\le N$ and $\mc{S}\subseteq \mb{R}/\mb{Z}$ be a $\delta$ separated set. Let $a_j$ be a sequence of complex numbers with $|a_j|\le j^{-1/2}$.
We say that $\Theta$ and $\Theta'$ in $\mc{S}$ are $\rho$-correlated if 
\[\max_{\substack{u_j\in \{\on{Re}(a_j \cdot e(j\Theta)), \on{Im}(a_j \cdot e(j\Theta))\}\\v_j\in \{\on{Re}(a_j \cdot e(j\Theta')), \on{Im}(a_j \cdot e(j\Theta'))\}}}\Big|\sum_{j=N+1}^{N+M}u_jv_j\Big|\ge \rho.\]

For each $\Theta$ in $\mc{S}$, there are at most 
\[\lesssim \rho^{-2} \cdot (\delta^{-1} + |\mc{S}|) \cdot MN^{-2}\]
many $\rho$-correlated phases $\Theta'$ in $\mc{S}$.
\end{lemma}
\begin{proof}
Consider $a_{N+1},\ldots,a_{N+M}$ which are complex numbers with $|a_j|\le j^{-1/2}$. Note that 
\begin{align*}
\Big|\sum_{j=N+1}^{N+M}&\on{Re}(a_j \cdot e(j\Theta)) \cdot \on{Re}(a_j \cdot e(j\Theta'))\Big|= \Big|\sum_{j=N+1}^{N+M}\frac{(a_j \cdot e(j\Theta) + \ol{a_j} \cdot e(-j\Theta))}{2} \cdot \frac{(a_j \cdot e(j\Theta') + \ol{a_j} \cdot e(-j\Theta'))}{2}\Big|\\
&\qquad\qquad\qquad\le \sup_{\substack{w_j \in \{a_j^2, |a_j|^2, \ol{a_j}^2\}\\\wt{\Theta}\in \{\pm \Theta \pm \Theta'\}}} \Big|\sum_{j=N+1}^{N+M}w_j \cdot e(j\cdot \wt{\Theta})\Big|.
\end{align*}
We may bound the sums of, e.g., $\on{Re}(a_j \cdot e(j\Theta)) \cdot \on{Im}(a_j \cdot e(j\Theta'))$, by the same quantity.  Note also that $|w_j|\le j^{-1}$ and therefore $\sum_{j=N+1}^{N+M}|w_j|^2\le M \cdot N^{-2}$. 

Fix $\Theta\in \mc{S}$. For an alternate $\Theta'\in \mc{S}$ to be $\rho$-correlated with $\Theta$, it must be that
\[\sup_{\substack{w_j \in \{a_j^2, |a_j|^2, \ol{a_j}^2\}\\\wt{\Theta}\in \{\pm \Theta \pm \Theta'\}}} \Big|\sum_{j=N+1}^{N+M}w_j \cdot e(j\cdot \wt{\Theta})\Big| \ge \rho.\]
By \cref{thm:large-siev}, there are at most 
\[\lesssim \rho^{-2} \cdot (\delta^{-1} + |\mc{S}|) \cdot MN^{-2}\]
such points. 
\end{proof}

We now prove our decorrelation estimate. Define $H:\mb{C}\to [0,1]$ via $H(z)$ to be a nonnegative $1$-bounded smooth bump function with $H(z) = 1$ for $|z| \le 1$ and $H(z) = 0$ for $|z|\ge 2$. We define 
\[G(w_1,\ldots,w_{\ell_k}) = H\Big(k^2 \cdot (w_1 + \cdots + w_{\ell_k})\Big) \cdot \prod_{j=1}^{\ell_k}H\Big(\delta_k^{-1/4} \cdot \sum_{r=1}^{j}w_{r}\Big).\]

\begin{lemma}\label{lem:decouple}
Let $z = e(\Theta)$ and $z' = e(\Theta')$. We define
\[\rho = \sup_{1\le j< \ell_k}\max_{\substack{u_j\in \{\on{Re}(a_j \cdot e(j\Theta)), \on{Im}(a_j \cdot e(j\Theta))\}\\v_j\in \{\on{Re}(a_j \cdot e(j\Theta')), \on{Im}(a_j \cdot e(j\Theta'))\}}}\Big|\sum_{n = r_{k,j}}^{r_{k,j+1}-1}u_jv_j\Big|.\]
Then if we write $\mbf{Q}_k(z) = (Q_{k,1}(z), \ldots,Q_{k,\ell_k}(z))$ we have 
\begin{align*}
\mb{E}[G(\mbf{Q}_k(z)) \cdot G(\mbf{Q}_k(z'))]
= \mb{E}[G( \mbf{Q}_k(z))] \cdot \mb{E}[G(\mbf{Q}_k(z')] + O(N_k^{-1/2} \cdot (\log N_k)^{O(1)} + \rho^{1/4} \cdot (\log N_k)^{O(1)}).
\end{align*}
\end{lemma}
\begin{proof}
Recall the definition of $\wt{Q}$ from \eqref{eq:wtq-def} and define $\wt{\mbf{Q}}_k$ analogously to $\mbf{Q}_k$.  
\cref{lem:lindberg-pairs} shows 
\begin{align*}
|\mb{E}[G(\mbf{Q}_k(z)) \cdot G(\mbf{Q}_k(z'))] - \mb{E}[G(\wt{\mbf{Q}}_k(z)) \cdot G(\wt{\mbf{Q}}_k(z'))]| \leq  N_k^{-1/2} \cdot (\log N_k)^{O(1)}
\end{align*}
where we used the bounds $\ell_k \leq (\log N_k)^{O(1)}$, $\delta_k^{-1/4} \leq \log k$ and $k^2 \leq (\log N_k)^{O(1)}$ and \cref{lem:lip} to control the third derivative.
We may analogously compare $\mb{E}[G(\mbf{Q}_k(z))]$ to $\mb{E}[G(\wt{\mbf{Q}}_k(z))]$.  Therefore it is sufficient to consider $\wt{\mbf{Q}}_{k}$ instead of $\mbf{Q}_k$. 

We now define 
\[{Q}_{k,j}^\circ(z) =  \sum_{n = r_{k,j}}^{r_{k,j+1}-1}\wt{g}_na_nz^{n};\]
here $\wt{g}_n$ are just an independent set of normal standard Gaussian variables. Define $\mbf{Q}^\circ_k$ analogously. Applying \cref{lem:coupling} iteratively we have that 
\begin{align*}
\Big|\mb{E}[G(\wt{\mbf{Q}}_k(z)) \cdot G(\wt{\mbf{Q}}_k(z'))] - \mb{E}[G(\wt{\mbf{Q}}_k(z)) \cdot G(\mbf{Q}^\circ_k(z'))] \Big|\le (\log N_k)^{O(1)} \cdot \rho^{1/4}. 
\end{align*}

However as $G(\wt{\mbf{Q}}_i(z))$ and $G(\mbf{Q}^\circ_i(z'))$ are independent we have the desired result.
\end{proof}

\subsection{Completing the proof}\label{subsec:comp}
We are now ready to complete the proof of \cref{thm:main}.
\begin{proof}[{Proof of \cref{thm:main}}]
Fix $\eps > 0$ to be sufficiently small.  We consider $k$ minimal such that $\delta_i\le \eps$ and $i\ge \eps^{-1}$ and set for all $n\le N_i$ that $a_n = 0$. As this only changes a bounded number of coefficients, this does not affect the convergence of any particular point.  By \cref{lem:sparsify-to-r}  and \cref{lem:deriv-controlled}, we may assume $\mc{E}_1^c \cap \mc{E}_2^c$ holds.  By \cref{lem:convergence-criterion}, we simply need to see that $\mc{A}$ is a nonempty set.

We first claim that 
\begin{equation} \label{eq:A_k-nested}
    \mc{A}_k + [-N_{k}^{-1} \cdot (\log N_k)^{-5},N_{k}^{-1} \cdot (\log N_k)^{-5}] \supseteq \mc{A}_{k+1} +[-N_{k+1}^{-1} \cdot (\log N_{k+1})^{-5},N_{k+1}^{-1} \cdot (\log N_{k+1})^{-5}]\,.
\end{equation}
To prove this, note first that $\mc{A}_{k+1}\subseteq \mc{A}_{k} + [-N_{k}^{-1} \cdot (\log N_{k})^{-8},N_{k}^{-1} \cdot (\log N_{k})^{-8}]$ by construction. Therefore it is sufficient to note 
\[1 \ge (\log N_{k})^{-3} + \frac{N_k\cdot (\log N_{k})^{-5}}{N_{k+1} \cdot (\log N_{k+1})^{-5}}\]
which proves \eqref{eq:A_k-nested}.  This shows the sets $\mc{A}_k + [-N_{k}^{-1} \cdot (\log N_k)^{-5},N_{k}^{-1} \cdot (\log N_k)^{-5}]$ are a nested sequence of compact sets.  Thus, to prove that $\mc{A} \neq \emptyset$ it is sufficient to prove that $\mc{A}_{k}\neq \emptyset$ for all $k$, by the finite intersection property.

We now recall that since we set $a_n =0$ for all $n \leq N_i$ that we have $|\mc{A}_i| = N_i$.  We proceed by induction and prove that $|\mc{A}_\ell|\ge N_{\ell}^{1-10^{-4}}$ for all $\ell\ge i$. The key claim will be that 
\begin{equation}\label{eq:second-moment-application}
\mb{P}[|\mc{A}_{\ell+1}|\ge N_{\ell+1}^{1-10^{-4}}| |\mc{A}_{\ell}|\ge N_{\ell}^{1-10^{-4}}] \ge 1 - N_{\ell}^{-1/20}\,.\end{equation}
Applying the union bound then gives \[ \mb{P}(\mc{A} \neq \emptyset) \geq 1 - O(N_i^{-1/20}) \geq 1 - \eps\]
which would complete the proof. 

Seeking to prove \eqref{eq:second-moment-application}, define \[\mc{A}_{s+1}^{\ast} = \Big\{\theta \in \mc{S}_{k+1}:  \exists~\varphi \in \mc{A}_{s}\text{ with }\Big|\varphi- \theta\Big|\le N_{s}^{-1} \cdot (\log N_{s})^{-8}.\Big\}\] and recall by the definition of $\mc{A}_{s+1}$ at \eqref{eq:A-k-def} that $\mc{A}_{s+1} = \mc{A}_{s+1}^\ast \cap \mc{G}_{s+1}$. Observe that 
\[|\mc{A}_{s + 1}^{\ast}| \asymp  |\mc{A}_{s}| \cdot M_{s + 1} \cdot (\log N_s)^{-8}.\]
We define 
\[X = \sum_{z\in \mc{A}_{s + 1}^{\ast}}G((Q_{s,1}(z),\ldots,Q_{s,\ell_{s}}(z)).\]
Note that $G((Q_{s,1}(z),\ldots,Q_{s,\ell_{s}}(z))$ is nonnegative, $1$-bounded with $\mbm{1}[\theta \in \mc{G}_{s+1}] \geq G((Q_{s,1}(z),\ldots,Q_{s,\ell_{s}}(z))$.  In particular, this implies $\mc{A}_{s+1} \geq X$.  Furthermore note by \cref{lem:one-point}, we have that 
\begin{equation}\label{eq:good-pt-LB-use}
\mb{E}[X|\mc{A}_s]\ge e^{-O((\log\log N_{s})^2 \cdot \delta_{s})} \cdot |\mc{A}_{s + 1}^{\ast}| \ge |\mc{A}_s| \cdot M_{s+1}^{1-10^{-5}}.\end{equation}

We now will upper bound the second moment.  Set $\rho = |\mc{A}_s|^{-1/3}$.  Note that $\mc{A}_{s+1}^\ast$ is a $\delta = N_{s+1}^{-1}$ separated set.  By \cref{lem:corr-point}, the number of pairs of $\rho$-correlated points is at most \[ \lesssim |\mc{A}_{s+1}^\ast| \cdot \ell_s \cdot \rho^{-2} (\delta^{-1} + |\mc{A}_{s+1}^\ast|) \cdot N_s^{-1} \lesssim |\mc{A}_s|^{2/3}\cdot \ell_s\cdot M_s\cdot |\mc{A}_{s+1}^\ast|\lesssim |\mc{A}_s|^{2/3}\cdot M_s^2\cdot |\mc{A}_{s+1}^\ast|\,. \]
Here we use that as $M\le N$ in \cref{lem:corr-point}, we may note that $MN^{-2}\le N^{-1}\le N_s^{-1}$ and that $|\mc{A}_{s+1}^\ast| \le \delta^{-1}\le N_{s+1}$.

By \cref{lem:decouple}, we then may upper bound 

\begin{align*}
    \mb{E}[X^2|\mc{A}_s] &\le (\mb{E}[X|\mc{A}_s])^2 + O((\log N_s)^{O(1)} \cdot (|\mc{A}_s|^{-1/12} + N_{s}^{-1/2})) \cdot |\mc{A}_{s + 1}^{\ast}|^2 + |\mc{A}_s|^{2/3} M_k^3 |\mc{A}_{s+1}^\ast| \\
    &\leq (\mb{E}[X|\mc{A}_s])^2 + |\mc{A}_{s+1}^\ast|^{2-{1/14}}
\end{align*}
where in the second inequality we assumed $s$ was large enough.  If we take $\lambda = M_{s+1}^{10^{-5}}$ we have \begin{align*}
    \mb{P}\left( |X - \mb{E}[X\,|\,\mc{A}_s]| \geq \lambda  \mb{E}[X\,|\,\mc{A}_s] \right) \leq \lambda^{-2} \frac{|\mc{A}_{s+1}^\ast|^{2 - 1/14}}{\mb{E}[X\,|\,\mc{A}_s]^2} \leq \frac{M_{s+1}^{O(1)}}{|A_s|^{1/14}}\,.
\end{align*}
This confirms \eqref{eq:second-moment-application}, completing the proof.
\end{proof}

\section{Proof of \texorpdfstring{\cref{thm:dim}}{Theorem 1.2}}\label{sec:Hausdorff}
The proof of \cref{thm:dim} relies on exactly the probabilistic tools to prove \cref{thm:main}.  Already the proof of \cref{thm:main} implicitly gives many convergent points.  To upgrade \cref{thm:main} to \cref{thm:dim}, we convert the point set $\mc{A}_s$ into an appropriate \emph{Frostman measure}, a standard tool for estimating Hausdorff dimension.  We will use the ``easy'' direction of Frostman's lemma.
\begin{lemma}\label{lem:Frostman}
Fix $\tau \in [0,1]$. Suppose $\mc{S}\subseteq [0,1]$ is a Borel set and that there is a Borel measure $\mu$ with $\mu(\mc{S}) > 0$ such that there is $C\ge 1$ such that for any $[a,b]\subseteq [0,1]$ then 
\[\mu([a,b])\le C \cdot |b-a|^{\tau}.\]
The $\mc{S}$ has Hausdorff dimension at least $\tau$.
\end{lemma}

We now set up the proof of \cref{thm:dim}. We will construct the ``Frostman'' measure via making the branching process structure of convergent points more explicit and defining the measure by pushing the measure down the generations of this branching process. 

\subsection*{Setting up the recursive structure}

Recall that  $\mc{S}_k = \Big\{\frac{j}{N_k}: j\in [N_k]\Big\}.$ 
Given $\theta \in \mc{S}_k$, we define the \emph{children} of $\theta$ by 
\[\on{Ch}(\theta) = \left\{\varphi \in \mc{S}_{k+1} : \Big|\theta - \varphi\Big|\le N_k^{-1} \cdot (\log N_k)^{-8}\right\}.\]
 For $\varphi \in \mc{S}_{k+1}$ we define the \emph{parent} of $\varphi$ to be the unique $\theta \in \mc{S}_k$ so that $\varphi \in \on{Ch}(\theta)$; we write $\on{Par}(\varphi) = \theta$ for the parent of a point. 
Note that a parent $\theta \in \mc{S}_k$ has $|\on{Ch}(\theta)| \asymp  M_{k+1} \cdot (\log N_{k})^{-8}$.
 It will also be useful to discuss the \emph{grand-children} of a node $\theta \in \mc{S}_k$ defined as \[\on{Gch}(\theta) = \bigcup_{\varphi \in \on{Ch}(\theta)}\on{Ch}(\varphi) \subset \mc{S}_{k+2}\,.\] 

Rather than working directly with $\mc{A}_s$ as in the proof of \cref{thm:main}, we will need a more robust notion.  We will define our ``healthy'' points $(\mc{H}_j)_{j \geq 1}$ by initializing $\mc{H}_1 = \mc{S}_1$ and inductively continuing as follows: \begin{equation}
    \mc{H}_{\ell} = \{ \theta \in \mc{S}_\ell \cap \mc{G}_\ell : \on{Par}(\theta) \in \mc{H}_{\ell - 1} \wedge |\{ \varphi \in \on{Ch}(\theta) \cap \mc{G}_{\ell+1} : |\on{Ch}(\varphi) \cap \mc{G}_{\ell+2}| \geq M_{\ell+2}^{1- \tau/2} \}| \geq M_{\ell+1}^{1-\tau}\} \,.
\end{equation}
The first condition of $\on{Par}(\theta) \in \mc{H}_{\ell - 1}$ will ensure that healthy points have a tree-like structure, which will allow us to define the measure by pushing it down to the next layer of healthy nodes.  The second condition can be interpreted as saying that not only are a large number of children of $\theta$ good, but a large number of children of $\theta$ also have a large number of \emph{their} children being good.

In analogy to $\mc{A}$ from \cref{lem:convergence-criterion}, we define \begin{equation} \label{eq:H-subset-A}
    \mc{H} = \bigcap_{k \geq 1} \left(\mc{H}_k + [-N_k^{-1}(\log N_k)^{-5}, N_k^{-1}(\log N_k)^{-5}] \right) \subset \mc{A}\,.
    \end{equation}

\subsection*{An inductive step: ensuring health points have healthy children}

We will use the same probabilistic toolkit used to prove \cref{thm:main} in order to show that healthy nodes beget healthy children.  

\begin{lemma}\label{lem:propagation-healthy}
    Let $\theta \in \mc{S}_\ell$ and suppose $U \subset \on{Ch}(\theta)$ with $|U| \geq M_{\ell+1}^{1 - \tau/2}$.  Then \begin{equation*}
        \mb{P}\left( \left|\{\varphi \in U : |\on{Ch}(\varphi) \cap \mc{G}_{\ell+2}| \geq M_{\ell+2}^{1-\tau/2} \} \right| \geq M_{\ell+1}^{1 - \tau} \right) \geq 1 - M_{\ell+2}^{-\Omega(\tau)} \,.
    \end{equation*}
\end{lemma}
\begin{proof}
    Set $V = \bigcup_{\varphi \in U} \on{Ch}(\varphi)$ and note that $|V| \asymp |U| \cdot M_{\ell + 2} (\log N_{\ell + 2})^{-8}$.  Mimicking the proof of \cref{thm:main}, set \[X = \sum_{z \in V} G(\mbf{Q}_{\ell+2}(z)) \] and note that $|V| \cap \mc{G}_{\ell+2} \geq X$.  For $\rho = M_{\ell+2}^{-1/5}$, we may bound the number of $\rho$-correlated pairs in $|V|$ using \cref{lem:corr-point} by \begin{equation*}
        \lesssim |V| \cdot \rho^{-2} \cdot N_{\ell+2} \cdot N_{\ell+1}^{-1} \leq |U| \cdot M_{\ell+2}^{8/5}\,.
    \end{equation*}
    Applying \cref{lem:decouple} we may then bound \begin{equation*}
        \mb{E}\left[ (X - \mb{E}[X])^2 \big| U\right] \leq |V|^2 M_{\ell+2}^{-1/{20}} + |U| \cdot M_{\ell+2}^{8/5}\,.
    \end{equation*}
    By Chebyshev's inequality, this implies \[\mb{P}\left(X \geq |V| M_{\ell+2}^{-\tau/10}\right) \geq 1 - M_{\ell+2}^{-\Omega(\tau)}\,.\]
    On this event, we have that at least $|U| M_{\ell+2}^{-\tau/10} \cdot (\log N_\ell)^{-O(1)} \geq M_{\ell+1}^{1-\tau}$ many elements of $U$ have the desired number of good children.
\end{proof}

\subsection*{Defining the Frostman measure} 
We will define a probability measure $\nu$ iteratively on $\mb{R} / \mb{Z} \cup \{*\}$ where $\{*\}$ is an isolated point that we treat as a sink state.  We define $\nu_1$ to be uniform on $\mc{H}_1 + [-N_1^{-1} (\log N_1)^{-5},N_1^{-1} (\log N_1)^{-5}].$  We will always maintain that \[\on{supp}(\nu_\ell) \subset \mc{H}_\ell + [-N_\ell^{-1} (\log N_\ell)^{-5},N_\ell^{-1} (\log N_\ell)^{-5}\,.\]  We then define $\nu_{\ell+1}$ on $\mc{H}_{\ell+1} + [-N_{\ell+1}^{-1} (\log N_{\ell+1})^{-5},N_{\ell+1}^{-1} (\log N_{\ell+1})^{-5}]$ as follows.  For $\theta \in \mc{H}_\ell$, set \[ U = \{\varphi \in \on{Ch}(\theta) : |\on{Ch}(\varphi) \cap \mc{G}_{\ell+2}| \geq M_{\ell+2}^{1 - \tau/2}  \} \] 
and recall that $|U| \geq M_{\ell+1}^{1 - \tau}$ since $\theta \in \mc{H}_\ell$.  We then split the mass of $\theta +[-N_\ell^{-1} (\log N_\ell)^{-5},N_\ell^{-1} (\log N_\ell)^{-5}]$ into $|U|$ many parts. For $\varphi \in U \cap \mc{H}_{\ell+1}$, assign the portion of this mass uniformly to the interval $\varphi+[-N_{\ell+1}^{-1} (\log N_{\ell+1})^{-5},N_{\ell+1}^{-1} (\log N_{\ell+1})^{-5}]$. For $\varphi \in U \cap \mc{H}_{\ell + 1}^c$, assign this mass to $\{*\}$.  Note that by \eqref{eq:A_k-nested}, each interval is assigned mass by at most one previous interval.  We define $\nu$ to be the limit of $\nu_{\ell}$.

We are now ready to put the pieces together to complete the proof. 

\begin{proof}[Proof of \cref{thm:dim}]
    Fix $\tau, \eps > 0$. Let $i$ be such that $\delta_i$ is sufficiently small with respect to $\tau$ and $\eps$.
 
    As in the proof of \cref{thm:main}, we may assume $\mc{E}_1^c \cap \mc{E}_2^c$ holds by \cref{lem:sparsify-to-r} and \cref{lem:deriv-controlled}.  We may again adjust $a_n = 0$ for all $n \leq N_{i+2}$ without affecting convergence; this implies that $\mc{H}_i = \mc{S}_i$.  We define the measure $\nu$ as above and note that $\nu$ is supported on $\mc{H} \cup \{*\}$.  By \cref{lem:convergence-criterion} and \eqref{eq:H-subset-A}, all points in $\mc{H}$ are convergent points.  To complete the proof, we need only show that $\nu$ assigns positive mass to $\mc{H}$ and that $\mc{H}$ is a $1 - O(\tau)$ Frostman measure.  

    For the former, note that by \cref{lem:propagation-healthy} and Markov's inequality, the expected fraction of mass set to $\{*\}$ at level $\ell$ is $M_{\ell}^{-\Omega(\tau)}$ for $\ell \geq i$ with probability $1 - M_{\ell}^{-\Omega(\tau)}$.  Since $\sum_{j > i} M_j^{-\Omega(\tau)} < \eps$ (for $i$ large enough), we have that with probability at least $1 - \eps$ the mass assigned to $\{*\}$ is at most, say, $1/2$. Since $\nu$ is a probability measure, this shows $\nu(\mc{H}) \geq 1/2.$  
    
    We now need to show the Frostman property.  First note that for each interval of the form $[i/N_{\ell}, (i+1)/N_{\ell}]$ we have \begin{equation}\label{eq:nu-small-interval}
    \nu\left([i/N_{\ell}, (i+1)/N_{\ell}] \right) \leq \nu_{\ell}([i/N_{\ell}, (i+1)/N_{\ell}]) \leq M_{\ell}^{-1 + \tau} \cdots M_{1}^{-1 + \tau} \lesssim N_{\ell}^{-1 + \tau}
    \end{equation}
    where the second inequality is since at level $k$ we split into at least $M_{k}^{-1 + \tau}$ many intervals and $[i/N_{\ell}, (i+1)/N_{\ell}]$ intersects at most $1$ interval from $\mc{H}_\ell +[-N_\ell^{-1} (\log N_\ell)^{-5},N_\ell^{-1} (\log N_\ell)^{-5}]$.  For an interval $I \subset \mb{R}/\mb{Z}$, choose $\ell$ so that $N_{\ell}^{-1} \leq |I| \leq N_{\ell-1}^{-1}.$  Then by losing at most a factor of $3$, we may assume that $I$ is of the form $[\frac{i}{N_\ell}, \frac{j}{N_{\ell}}]$.  Further, since $|I| \leq N_{\ell - 1}^{-1}$ we may also assume $1 \leq j - i \leq 3\cdot M_{\ell}\,.$  We may then bound \begin{align*}
        \nu(I) \lesssim (j-i) N_{\ell}^{-1 + \tau} \lesssim M_\ell N_{\ell}^{-1 + \tau} \leq \left(\frac{M_\ell}{N_{\ell}}\right)^{1 - 2\tau} \lesssim |I|^{1 - 2\tau}\,.
    \end{align*}
    This shows that with probability at least $1 - \eps$, the set $\mc{H}$ supports a $1 - 2\tau$ Frostman measure.  Since $\eps$ is arbitrary, this shows that for each fixed $\tau$, $\mc{H}$ almost-surely has Hausdorff dimension at least $1 - 2\tau$ by \cref{lem:Frostman}.  Applying this statement for a countable sequence of $\tau$ tending to zero and recalling that almost-surely $P$ converges on $\mc{H}$ by \cref{eq:H-subset-A} and \cref{lem:convergence-criterion} completes the proof. 
\end{proof}

\bibliographystyle{amsplain0}
\bibliography{main.bib}
\end{document}